\documentclass[11pt]{article}
\usepackage{graphicx}
\usepackage{latexsym}
\usepackage{amsmath}
\usepackage{amssymb}
\usepackage{amsthm}
\usepackage{enumitem}
\usepackage{caption}
\usepackage{float}
\usepackage{mathpazo}
\usepackage{hyperref}
\usepackage{tikz}
\usepackage{multicol}
\usepackage{mathtools}
\usepackage{authblk} 
\def \N {{\mathbb N}}
\def \Z {{\mathbb Z}}
\def \G {{\mathcal G}}

\def \I {{\mathcal I}}

\def \E {{\mathcal E}}

\def \V {{\mathcal V}}

\def \P {{\mathcal P}}
\def \v {{\mathbf v}}

\def \Ga {{\Gamma}}
\def \la {{\lambda}}

\def \u {{\mathfrak u }}

\newtheorem{theorem}{Theorem}[section]
\newtheorem{cor}[theorem]{Corollary}
\newtheorem{lemma}[theorem]{Lemma}
\newtheorem{pro}[theorem]{Proposition}
\newtheorem{rem}[theorem]{Remark}

\newtheorem{definition}[theorem]{Definition}

\newtheorem{ex}[theorem]{Example}

\usepackage[colorinlistoftodos,prependcaption,textsize=tiny]{todonotes}
\usepackage{xargs} % Use more than one optional parameter in a new commands
\newcommandx{\unsure}[2][1=]{\todo[linecolor=red,backgroundcolor=red!25,bordercolor=red,#1]{#2}}
\newcommandx{\change}[2][1=]{\todo[linecolor=blue,backgroundcolor=blue!25,bordercolor=blue,#1]{#2}}
\newcommandx{\newidea}[2][1=]{\todo[linecolor=green,backgroundcolor=green!25,bordercolor=green,#1]{#2}}
\newcommandx{\improvement}[2][1=]{\todo[linecolor=pink,backgroundcolor=pink!25,bordercolor=pink,#1]{#2}}
\usepackage{setspace}
\usepackage{natbib}
\usepackage[english]{babel}
\usepackage[euler]{textgreek}
\usepackage{stmaryrd}
\usepackage{mathtools}
\usepackage{pgf,tikz}
\usepackage{tikz}
\usepackage{tkz-graph}
\usetikzlibrary{positioning}
\usepackage{array}
\date{}

\title{On the Generating Graph of Finite Abelian Groups}
\author[1]{Kavita Samant}
\author[2]{A. Satyanarayana Reddy}
\affil[1,2]{{\normalsize Department of Mathematics\\Shiv Nadar Institution of Eminence, Delhi-NCR, India}\\ 
E-mail id(s): ks299@snu.edu.in; satya.a@snu.edu.in}
\date{}
\begin{document}
\maketitle
\begin{abstract}
    The generating graph $\Gamma(G)$ of a group $G$ is the graph whose vertex 
    set is $G$, where two distinct vertices are adjacent if and 
    only if they generate $G$. In this paper, we systematically study the 
    structure of generating graphs of finite abelian groups (non-cyclic)  and 
    determine the set of all generating pairs. Moreover, we give some structural characterizations, in particular, we determine conditions under which $\Gamma(G)$ is 
    regular, characterize when the isolated vertices form a subgroup, and 
    establish necessary and sufficient conditions for two non-isomorphic 
    finite abelian groups $G$ and $H$ to satisfy $\Gamma(G)\cong 
    \Gamma(H)$. Furthermore, 
    we compute the spectra of the adjacency and Laplacian matrices of these 
    graphs.
\end{abstract}
  
\section{Introduction}

Graphs associated with groups provide a way to study algebraic structures using tools from graph theory. In these constructions, the elements of a group, the set of subgroups, the set of conjugacy classes, etc., form the vertices, and adjacency is defined through some property of the group. Examples include Cayley graphs, commuting graphs, generating graphs, power graphs, and prime graphs. Each graph highlights different aspects of the underlying group’s structure. We refer to the paper~\cite{cameron} by P.~J.~Cameron for further background on graphs defined on groups. In this paper, we study generating graphs of groups. For a given group $G,$ the generating graph is denoted by $\Ga(G),$ whose vertex set is $G$, and any two distinct vertices are adjacent if they generate $G$. These graphs are interesting mainly for \emph{2-generated groups} (groups that can be generated by two elements), as otherwise the graph is empty. 

In the study of 2-generated groups, the generating graphs are introduced (explicitly) and investigated in the papers~\cite{some,clique}. For certain classes of groups, such as finite simple groups \cite{Lieback,Guralnick}, solvable groups~\cite{generating,diameter}, and nilpotent groups~\cite{Connectivity}, their graph properties are studied extensively in the literature. For instance, for finite solvable groups, it is known that the subgraph $\Delta(G),$ obtained by removing isolated vertices from $\Gamma(G)$, is connected. Moreover, its diameter satisfies $\mathrm{diam}(\Delta(G)) \leq 3$, and this bound is sharp, although in some important cases, including all finite abelian groups, the bound improves to $\mathrm{diam}(\Delta(G)) \leq 2$. More such graph properties for abelian groups can be seen in the paper~\cite{Connectivity}. The generating graph of an infinite abelian group and its connectivity has been explored in the paper~\cite{infinite}.

In recent years, spectral theory has gained interest in the study of graphs associated with algebraic structures. Spectral properties often link the algebraic nature of the group with combinatorial and geometric features of the graph. However, like other group associated graphs, the spectral study of generating graphs is not extensively studied.  In this work, we study the generating graph of finite abelian groups from both a structural and spectral viewpoint. Previously, the spectral properties of the generating graphs have been thoroughly explored in the papers~\cite{Me1,Me2,Me} for dihedral groups, dicyclic groups, and finite cyclic groups, respectively.  The study is significant because it bridges group theory with spectral graph theory, with potential applications in network design, algebraic combinatorics, and in understanding the relationship between a group's algebraic structure and its graph-theoretic representations.

\paragraph*{Our contribution:} In this paper, we study the generating graph $\Gamma(G)$ of finite abelian 
groups $G$ (non-cyclic). We explicitly determine the set of generating pairs, describe the graph 
structure.

A central theme is that abelian $p$-groups serve as the fundamental building 
blocks: the generating graph of any finite abelian group is completely 
determined by the generating graphs of its Sylow subgroups. This mirrors the 
classical primary decomposition of abelian groups, and reduces most questions 
about $\Gamma(G)$ to the $p$-group case.

A natural question is which elements fail to generate, that is, which 
vertices are isolated in $\Gamma(G)$. It is straightforward to see 
that every non-identity element of the Frattini subgroup $\Phi(G)$ is 
isolated. However, the converse fails in general. For finite abelian groups that are not $p$-groups, we show that the isolated vertices occur outside 
$\Phi(G)$. We explicitly identify all isolated vertices and characterize when they form a subgroup of $G$. Moreover, we determine necessary and sufficient conditions for $\Gamma(G)$ to be regular, and show that regularity is completely controlled by the cyclic or non-cyclic nature 
of the Sylow subgroups of $G$.
Moreover, we investigate the graph isomorphism for abelian groups: we basically determine necessary and 
sufficient conditions for two non-isomorphic finite abelian groups $G$ and $H$ 
to satisfy $\Gamma(G) \cong \Gamma(H)$. 
Lastly, we compute the spectra of the adjacency and Laplacian matrices.
\\
\vspace{0.1cm}\\
This paper is organized into six sections. In Section~\ref{sec:2}, we review some fundamental results on abelian groups and some important notions of graphs. In Section~\ref{sec:3}, we determine the set of all generating pairs of the groups. In Section~\ref{sec:4}, we review and generalize a few definitions and results on the generating graph of cyclic groups. Furthermore, in Section~\ref{sec:5} , we investigate the structure of the generating graph of abelian groups and present some structural characterizations. Lastly, in Section~\ref{sec:6}, we determine the adjacency spectrum and Laplacian spectrum of the associated graph. 

\section{Preliminaries}\label{sec:2}
Unless otherwise specified, all groups under consideration are finite. 

A \emph{minimal generating set} of a group $G$ is a generating set $X$ 
such that no proper subset of $X$ generates $G$. In general, a group 
can have minimal generating sets of different sizes. However, for 
$p$-groups this size is an invariant, determined by the Frattini 
subgroup. The \emph{Frattini subgroup} $\Phi(G)$ of a group $G$ is the intersection of all maximal subgroups of $G$. 
\begin{theorem}[Theorem 5.48,~\cite{rotman}]
    Let $G$ be a finite $p$-group. Then $G/\Phi(G)$ is isomorphic to a finite dimensional vector space over the finite field $\mathbb{F}_p.$
\end{theorem}
\begin{cor}[\cite{rotman}]\label{cor:imp}
    Let $G$ be a finite $p$-group. Then $\Phi(G)$ is contained in every 
    maximal subgroup of $G$, and there is a one-to-one correspondence 
    between maximal subgroups of $G$ and maximal subgroups of $G/\Phi(G)$. 
    In particular, the number of maximal subgroups of $G$ equals the number 
    of subgroups of index $p$ in the elementary abelian group $G/\Phi(G)$.
    \end{cor}

\begin{theorem}[Burnside's Basis Theorem, Theorem~5.50,~\cite{rotman}]
\label{thm:burnside}
Let $G$ be a finite $p$-group. Then any minimal generating set of $G$ has 
cardinality $\dim(G/\Phi(G))$, and $x \in G$ belongs to some minimal 
generating set if and only if $x \notin \Phi(G)$.
\end{theorem}

Now we state some fundamental results for finite abelian groups. 

 \begin{theorem}[Theorem 6.9,~\cite{rotman}]
Every finite abelian group $G$ is a direct sum of cyclic groups.
\end{theorem}
\begin{definition}
    If $G$ has a decomposition as a direct sum $G = \mathbb{Z}_{n_1} \times \mathbb{Z}_{n_2} \times \cdots \times 
    \mathbb{Z}_{n_k},$ where $n_k\mid n_{k-1}\mid\dots \mid n_1,$ then one says that $G$ has \emph{invariant factors} $(n_1,n_2,\dots,n_k)$. 
\end{definition}
 Let $d(G)$ denote the minimum size of a generating set of $G.$ The following theorem determines $d(G)$ for a finite abelian group.
%    \todo{check how to argue this}
\begin{theorem}[\cite{minabelian}]
    Let $G$ be a product of finite cyclic groups $$G= \Z_{n_1} \times \Z_{n_2}\times \dots \times \Z_{n_k},$$ where
    for all i, $n_i>1$. For every prime $p$, let $d_p=|\{i\leq  k : p\mid n_i\}|$ and
    $d(G) = \max\{d_p : p \text{ prime} \}.$
    Then the smallest size of a generating set of $G$ is $d(G).$
\end{theorem}

\begin{rem}
        By the above theorem, an abelian group $G$ has $d(G) = 2$ if and only if $G \cong 
        \mathbb{Z}_n \times \mathbb{Z}_m$ with $m \mid n$ and $m, n > 1$. 
        Throughout this paper, we work with groups of this form.
        \end{rem}

    \begin{pro}[\cite{rotman}]\label{pro:imp}
        Let $\pi(n)$ denote the set of distinct prime divisors of $n$.
        \begin{enumerate}
            \item For $G = \mathbb{Z}_n$,
            $$\Phi(\mathbb{Z}_n) = \bigcap_{p \in \pi(n)} p\mathbb{Z}_n,$$
            where $p\mathbb{Z}_n = \{pg : g \in \mathbb{Z}_n\}$ is the unique 
            subgroup of index $p$ in $\mathbb{Z}_n$.
        
            \item For a finite abelian group $G = \mathbb{Z}_{p_1^{e_1}} \times 
            \cdots \times \mathbb{Z}_{p_k^{e_k}}$ with distinct primes $p_i$,
            $$\Phi(G) = \Phi(\mathbb{Z}_{p_1^{e_1}}) \times \cdots \times 
            \Phi(\mathbb{Z}_{p_k^{e_k}}).$$
        \end{enumerate}
        \end{pro}

\subsection{Kronecker product of graphs and equitable partition}
Recall that the \emph{Kronecker product} $A\otimes B$ of two matrices $A = [a_{ij}]$ and $B, $ is the partitioned matrix $[a_{ij}B]$. The graph product, known as the Kronecker product of graphs, is derived from the
Kronecker product of matrices (for details see~\cite{kronecker}). 

 Let $\G_1 = (V_1, E_1)$ and $\G_2 = (V_2, E_2)$ be two graphs with the vertex set $V_i$ and the edge set $E_i$ ($i=1,2$). 
The \emph{Kronecker product} $\G_1 \otimes \G_2$, is the graph with the vertex set $V_1 \times V_2$, and $E(\G_1\otimes \G_2)=\left\{(u_1, v_1) (u_2, v_2):\, u_1u_2\in E_1 \text{ and } v_1v_2\in E_2  \right\}.$
This definition assumes that both $\G_1$ and $\G_2$ are simple graphs (that is, without loops or multiple edges).
The graphs $\G_1$ and $\G_2$ are called \emph{factors} of the product
$\G_1\otimes \G_2.$ Other popular names for the Kronecker product that have appeared in the literature are tensor
product, direct product, and cross product.

For a graph \(\G\) and a vertex \(v \in V(\G)\), the degree of \(v\) in \(G\) is denoted by \(\deg_\G(v)\). When the underlying graph is clear from the context, we simply write \(\deg(v)\). 
Now we recall some properties and results on the Kronecker product of graphs~\cite{kronecker}.

\begin{pro}\label{pro:per}
    The following properties hold for the Kronecker product of graphs $\G_1$ and $\G_2.$
    \begin{enumerate}
        \item $|E(\G_1\otimes \G_2)|=2|E(\G_1)||E(\G_2)|;$
        \item For any $(u,v)\in V(\G_1\otimes \G_2),$ 
        $$deg_{\G_1\otimes \G_2}(\left(u,v\right))=deg_{\G_1}(u)deg_{\G_2}(v).$$
    \end{enumerate}
\end{pro}
\begin{theorem}[Theorem 1,~\cite{kronecker}]\label{thm:kcon}
Let $\G_1$ and $\G_2$ be connected graphs. The product $\G_1\otimes \G_2$ is connected if and only if either $\G_1$ or $\G_2$ contains an odd cycle. 
\end{theorem}
\begin{theorem}[\cite{kronecker}] \label{thm:kadd}
    Let $A_1$ and $A_2$ be the adjacency matrices of the graphs $\G_1$ and $\G_2,$ respectively. The adjacency matrix of the graph $\G_1\otimes \G_2$ is $A_1\otimes A_2,$ the Kronecker product of matrices $A_1$ and $A_2.$
\end{theorem}
\begin{cor}\label{cor:kspec}
   If $\lambda_i,$ where $ i = 1, \ldots, m$ and $\mu_j,$ where $ j = 1, \ldots, n$ are adjacency eigenvalues of $\G_1$ and $\G_2$, respectively, then $\lambda_i \mu_j,$ where $ i = 1, \ldots, m$ and $ j = 1, \ldots, n$ are the adjacency eigenvalues of $\G_1 \otimes \G_2.$
\end{cor}

From~(2) of Proposition~\ref{pro:per}, the Kronecker product of regular graphs is regular, and one can easily compute the Laplacian spectrum in that case. Otherwise, the following result gives some partial eigenvalues of the Laplacian matrix. We adopt a notation $S(M)$ for the multiset of the eigenvalues of the matrix $M.$
 \begin{theorem}[Theorem 13,~\cite{barik}]\label{thm:klap}
Let $\G_1$ be a regular graph on $m$ vertices with regularity $r$ and $\G_2$ be any graph on $n$ vertices. Suppose that $S(L(\G_2)) = \{\mu_1, \dots, \mu_n\}$. Then
$$
     r\mu_j \in S(L(\G_1 \otimes \G_2))
$$
 for each $j = 1, 2, \dots, n$.
 \end{theorem}

\begin{rem}
    In this paper, we also consider the Kronecker product $\G_1 \otimes 
\G_2$ when one factor contains loops. We adopt the convention that a 
loop at a vertex $v$ in $\G_2$ contributes to adjacency in the product: 
an edge exists between $(u_1,v_1)$ and $(u_2,v_2)$ if $u_1u_2 \in 
E(\G_1)$ and $v_1v_2 \in E(\G_2)$, including the case $v_1 = 
v_2$ when $\Gamma_2$ has a loop at $v_1$. Further, we assume that a loop contributes as a single edge in the graph.
Under this convention, if $\G_1$ is simple then $\G_1 \otimes \G_2$ is 
simple (a loop at $(u,v)$ in the product would require a loop at $u$ in 
$\G_1$). Further, the degree formula $\deg_{\G_1 \otimes \G_2}((u,v)) = \deg_{\G_1}(u)\deg_{\G_2}(v)$ continues to hold. However, under the assumption that a loop contributes to a single edge in the graph, it implies that the edge formula $|E(\G_1\otimes \G_2)|=2|E(\G_1)||E(\G_2)|$ will be modified in the case when $\G_2$ has loops and $\G_1$ is simple as follows: 
\begin{equation}
    |E(\G_1\otimes \G_2)|=2|E(\G_1)||E^o(\G_2)|+|E(\G_1)||\mathcal{L}(\G_2)|,
\end{equation}
where $E^o(\G_2)$ denote non-loop edges and $\mathcal{L}(\G_2)$ denote loops of $\G_2.$ Moreover, the adjacency matrix of $\G_1 \otimes \G_2$ remains $A_1 \otimes A_2$. 
All spectral results stated above therefore, carry over without change.
\end{rem}

\subsection*{Equitable partition}
The concept of an equitable partition is a key tool in the computation of the eigenvalues of large matrices. Now we briefly recall the definitions and some results; for further details see~\cite{book}.

Let $M$ be an $n \times n$ matrix a square matrix whose rows and columns are indexed by elements of $C=\{1,2,\dots n\}.$ Let $\P=\{C_1,\ldots, C_k\}$ be a partition of $C$. Now we partitioned the matrix $M$ into blocks, according to $\P$ as follows:
$$\begin{pmatrix}
    C_{11}&C_{12}&\cdots&C_{1k}\\
    C_{21}&C_{22}&\cdots&C_{2k}\\
    \vdots&\vdots&\ddots&\vdots\\
    C_{k1}&C_{k2}&\cdots&C_{kk}
\end{pmatrix},$$
where $C_{ij}=M[C_i,C_j]$ is a submatrix of $M$ whose rows are indexed by $C_i$ and columns are indexed by $C_j$ for $i,j=1,2,\dots,k.$ The partition $\P$ is called \emph{equitable} if the row sum of each block $C_{ij}$  is a constant $b_{ij}$ and the associated matrix $B=(b_{ij})_{k\times k}$ is called a \emph{quotient matrix} of $M.$ The following is a well known result on an equitable partition of a matrix.

\begin{theorem}[\cite{book}]\label{thm:specA}
Let $B$ be the quotient matrix of a square matrix $M$ with respect to an equitable 
partition. Then every eigenvalue of $B$ is an eigenvalue of $M$.
\end{theorem}

\section{{Generating pairs}}\label{sec:3}
Let $d(G)$ denote the minimum size of a generating set of a group $G.$ In this section, we determine and formulate the set of generating pairs of finite abelian groups with $d(G)=2.$
Let $G$ be an abelian group of order $n.$ Let $n=p_1^{e_1}p_2^{e_2}\dots p_k^{e_k},$ with distinct primes $p_i$'s. The Euler totient 
function is denoted by $\varphi.$ Throughout, we consider abelian groups under componentwise addition.

\subsection{On abelian \texorpdfstring{$p$}{p}-groups}

Let us first determine the set of all generating pairs for abelian $p$-groups.
For a group $G$, the set of generating pairs is defined by
\[
  \mathrm{Gen}(G) = \{(x,y)\in G\times G \,:\, \langle x,y\rangle = G\}.
\]
Moreover, for each $x\in G$, we associate the set
\[
  N_x = \{y\in G \,:\,  \langle x,y\rangle = G\},
\]
which coincides with the neighbourhood of $x$ in the generating graph $\Gamma(G)$.  The following result is a consequence of the fundamental properties of finite abelian $p$-groups and is useful in our study. For the sake of completeness, we are proving it here.
\begin{pro}\label{pro:pgen}
    Let \( G \) be a finite abelian \( p \)-group with $d(G)=2.$ Let \( \{K_1, K_2, \dots, K_s\} \) be the set of all maximal subgroups of \( G \). Then, $s=p+1$ and  $$
        \mathrm{Gen}(G)=\{(x,y): x\in K_i,\; y\in K_j,\; i\neq j,\; \text{and } x,y\notin \Phi(G)\}.
        $$ Moreover, for any $x\in G\setminus \Phi(G),$  $|N_x|=\varphi(|G|).$
\end{pro}
\begin{proof}
    Let $G \cong \mathbb{Z}_{p^e} \times \mathbb{Z}_{p^t}$, where $p$ is a prime and $e, t\geq 1$. Since $G$ is a $p$-group and $d(G)=2$, apply Burnside’s Basis Theorem~\ref{thm:burnside}, we get for every $x \in G \setminus \Phi(G)$, there exists $y \in G$ such that $\langle x, y \rangle = G$. Thus, $\mathrm{Gen}(G)\neq \emptyset.$ Moreover, any $x\in \Phi(G)$ is a non generator, therefore any pair $(x,y)\in \mathrm{Gen}(G)$ implies that neither $x$ nor $y\in \Phi(G).$ 
    
    Since $G$ is a direct product of cyclic $p$-groups, $\Phi(G) = \Phi(\mathbb{Z}_{p^e}) \times \Phi(\mathbb{Z}_{p^t})$ (see Proposition~\ref{pro:imp}), and can be explicitly written as
    \begin{equation}
    \Phi(G) = \{(a, b) \in \mathbb{Z}_{p^e} \times \mathbb{Z}_{p^t} \,:\,  a = ip,\ b = jp,\ 0\leq i\leq p^{e-1}-1, 0\leq j\leq p^{t-1}-1 \}.
    \end{equation}
    Thus, $|\Phi(G)| = p^{e-1} \cdot p^{t-1} = p^{e + t - 2}$, and the quotient $G / \Phi(G)$ has order $p^2$, and isomorphic to $\mathbb{Z}_p \times \mathbb{Z}_p$. Also note that $G/\Phi(G)$ has exactly $p + 1$ subgroups of order $p$, which correspond to the maximal subgroups of $G$. Therefore, $G$ has exactly $p + 1$ maximal subgroups, each of index $p$ and order $p^{e + t - 1}$, and any two have intersection exactly $\Phi(G)$ (see Corollary~\ref{cor:imp}).

    Let $\mathcal {M}(G)=\{K_1,K_2,\dots,K_{p+1}\}$ be the set of all maximal subgroups of $G.$ We suppose $x\in K_i\setminus \Phi(G)$ and $y\in K_j\setminus \Phi(G).$ Then $\langle x,y\rangle\leq G.$ Notice that $\langle x,y\rangle= G,$ otherwise  $\langle x,y\rangle \subseteq K_r$ for some $1\leq r\leq (p+1),$ which contradicts that $x,y\notin \Phi(G).$ Therefore, we have
    $$
    \mathrm{Gen}(G)=\{(x,y): x\in K_i,\; y\in K_j,\; i\neq j,\; \text{and } x,y\notin \Phi(G)\}.
    $$ Moreover, $|K_i \setminus \Phi(G)|=\varphi(p^{e + t - 1})$ for any $1\leq i\leq (p+1).$ Thus for any $x\in G,$
    $
    |N_x|=p \varphi(p^{e + t - 1}) = \varphi(p^{e + t})=\varphi(|G|).
    $
    This completes the proof.
\end{proof}
\begin{rem}\label{rem:edge}
 For $G=\Z_{p^e}\times \Z_{p^t},$ the cardinality of $\mathrm{Gen}(G)$ is  $$|\mathrm{Gen}(G)|=p(p+1)(\varphi(p^{e+t-1}))^2.$$
\end{rem}
To determine explicitly all maximal subgroups of finite abelian $p$-groups, we first recall Goursat's lemma for groups.
\begin{pro}[\cite{Goursat}]\label{pro:goursat}
    Let \(G\) and \(H\) be arbitrary groups. Then the set \(\mathcal{S}\) of all subgroups of the direct product \(G \times H\) is in bijective correspondence with the set of all 5-tuples \((A,B,C,D,\Psi)\), where \(B \trianglelefteq A \leq G\), \(D \trianglelefteq C \leq H\), and
\(
\Psi:A/B \longrightarrow C/D
\)
is an isomorphism. Under this correspondence, the subgroup associated with the tuple \((A,B,C,D,\Psi)\) is
\begin{equation}\label{eq:iso}
K=\{(a,c)\in A\times C:\Psi(aB)=cD\}.
\end{equation} Moreover, $|A||D|=|K|=|B||C|.$
\end{pro}

Now, we obtain the following description of the maximal subgroups of $\mathbb{Z}_{p^e} \times \mathbb{Z}_{p^t}$.

\begin{pro}
        Let $p$ be a prime and let $e \geq t \geq 1$ be integers. The maximal subgroups of $\mathbb{Z}_{p^e} \times \mathbb{Z}_{p^t}$ are precisely the following $p+1$ subgroups
        \[
        \mathbb{Z}_{p^e} \times \langle p \rangle, \quad \langle p \rangle \times \mathbb{Z}_{p^t}, \quad K_\ell := \{(x,y) \in \mathbb{Z}_{p^e} \times \mathbb{Z}_{p^t} : y \equiv \ell x \pmod{p}\} \ \ (\ell = 1, \dots, p-1).
        \]
        \end{pro}
        
        \begin{proof}
        By Proposition~\ref{pro:goursat}, subgroups $K \leq \mathbb{Z}_{p^e} \times \mathbb{Z}_{p^t}$ correspond bijectively to the tuples $(A,B,C,D,\Psi)$. Since $\mathbb{Z}_{p^e}$ and $\mathbb{Z}_{p^t}$ are cyclic, each subgroup is determined uniquely by its order. 
        Suppose $|A| = p^\alpha$, $|B| = p^\beta$, $|C| = p^\gamma$, $|D| = p^\delta$ with $0 \leq \beta \leq \alpha \leq e$ and $0 \leq \delta \leq \gamma \leq t$. Every subgroup is normal in its original group, and $A/B \cong \mathbb{Z}_{p^{\alpha-\beta}}$, $C/D \cong \mathbb{Z}_{p^{\gamma-\delta}}$. An isomorphism $\Psi : A/B \to C/D$ exists if and only if $\alpha - \beta = \gamma - \delta=k$ (say), in which case the isomorphisms are indexed by units $\ell$ modulo $p^k$.
        
   Further, $|K| = |A||D| = p^{\alpha+\delta}$. Since $\Z_{p^e}\times \Z_{p^t}$ is a $p$-group, $K$ is maximal if and only if $[\Z_{p^e}\times \Z_{p^t}: K] = p$, that is,
        \begin{equation}\label{eq:alpha}
        \alpha + \delta = e + t - 1. 
        \end{equation}
        Since $\delta = \gamma - k$, Equation~\ref{eq:alpha} becomes $\alpha + \gamma = e + t - 1 + k$. As $\alpha \leq e$ and $\gamma \leq t$, we have $\alpha + \gamma \leq e + t$, so $k \in \{0, 1\}$.
        \begin{description}
        \item[Case \bf{$k = 0$}:] In this case, $A = B$, $C = D$, and $\Psi$ is trivial, then $\alpha + \gamma = e + t - 1$ with $\alpha \leq e$, $\gamma \leq t$ forces $(\alpha, \gamma) \in \{(e, t-1), (e-1, t)\}$, giving sugroups $\Z_{p^e} \times \langle p \rangle$ and $\langle p \rangle \times \Z_{p^t}$, respectively.
        
        \item[Case \bf{$k = 1$}:] In this case, $\alpha + \gamma = e + t$ with $\alpha \leq e$, $\gamma \leq t$ forces $\alpha = e$, $\gamma = t$, hence $\beta = e-1$, $\delta = t-1$. Here $\Psi : \mathbb{Z}_p \to \mathbb{Z}_p$ ranges over the $p-1$ nontrivial choices $\ell \in \mathbb{Z}_p^*$, and by expression~\ref{eq:iso} the corresponding subgroup is given by
        \[
        K_\ell =\{(g,h) \in \Z_{p^e} \times \Z_{p^t} : \Psi(g \bmod p) = h \bmod p\} = \{(x,y) : y \equiv \ell x \pmod{p}\}.
        \]
        \end{description}
        Altogether this gives $2 + (p-1) = p+1$ maximal subgroups.
        \end{proof}

\subsection{The general case}  
We now determine the generating pairs of $G = \mathbb{Z}_n \times \mathbb{Z}_m$, 
where $m \mid n$. Write
\[
  m = p_1^{t_1} p_2^{t_2} \cdots p_j^{t_j}, \qquad 
  n = p_1^{e_1} p_2^{e_2} \cdots p_k^{e_k},
\]
where $p_1, p_2, \ldots, p_k$ are distinct primes, $t_i \leq e_i$ for 
$1 \leq i \leq j$, and $j \leq k$. By using the fact that if $\gcd(a,b)=1,$ then $\Z_a\times \Z_b\cong \Z_{ab},$ so $G$ 
decomposes according to the following two cases.

\begin{enumerate}[label=(\alph*)]
    \item $k = j$: every prime dividing $n$ also divides $m$, and
    \[
      G \cong \prod_{i=1}^{k} 
      \bigl(\mathbb{Z}_{p_i^{e_i}} \times \mathbb{Z}_{p_i^{t_i}}\bigr)= \prod_{i=1}^{k} 
     S_{p_i},
    \]
   where each factor $\mathbb{Z}_{p_i^{e_i}} \times \mathbb{Z}_{p_i^{t_i}}$ 
    is the Sylow $p_i$-subgroup of $G$, denoted by $S_{p_i}$.
    
    \item $j < k$: the primes $p_{j+1}, \ldots, p_k$ divide $n$ but not $m$, 
    and
    \[
      G \cong \left(\prod_{i=1}^{j} 
      S_{p_i}\right) 
      \times \,\mathbb{Z}_s,
    \]
    where $s = p_{j+1}^{e_{j+1}} \cdots p_k^{e_k}$.
\end{enumerate}
Throughout, we work with these decompositions unless stated otherwise. 

\begin{theorem}\label{thm:agen}
Let $k = j$ and let $x = (a_1, a_2, \ldots, a_k)$ and 
$y = (b_1, b_2 ,\ldots, b_k) \in G$.
Then $\langle x, y \rangle = G$ if and only if 
$\langle a_i, b_i \rangle = S_{p_i}$ 
for all $1 \leq i \leq k$.
\end{theorem}
    \begin{proof}
        Let $G = S_{p_1} \times S_{p_2} \times \cdots \times S_{p_k},$ where each $S_{p_i}$ is non-cyclic.
         Suppose $\langle x, y \rangle = G$ but 
        $\langle a_i,b_i\rangle \neq S_{p_i}$ for some $i$. 
        Then $\langle a_i,b_i\rangle \neq S_{p_i}$ is a proper subgroup of 
        $S_{p_i}$, hence contained in some maximal subgroup $M \subsetneq S_{p_i}$. 
        But then $\langle x, y \rangle \subseteq S_{p_1} \times \cdots \times M 
        \times \cdots \times S_{p_k} \subsetneq G$, a contradiction.
        
    Suppose $\langle a_i,b_i \rangle = S_{p_i}$ 
        for all $1 \leq i \leq k$. Let $g = (c_1, c_2, \ldots, c_k) \in G$ 
        be arbitrary. Since $S_{p_i} = \langle a_i, b_i \rangle$, 
        each entry $c_i$ can be expressed in the product of
        $a_i,b_i$ and their inverses. Since $S_{p_i}$'s are 
        coprime-order factors and they commute, we can easily show that
        $g \in \langle x, y \rangle$. Hence $\langle x, y \rangle = G$.
        \end{proof}

The following result follows directly from the above theorem.
\begin{cor}
    For $k=j,$ the set of generating pairs of $G$ is given by $$\mathrm{Gen}(G)=\prod_{i=1}^{k}\mathrm{Gen}(S_{p_i}).$$ 
\end{cor}

\begin{theorem}\label{thm:bgen}
    Given $j<k.$ Let $x=(a_1,a_2,\ldots ,a_j,a_s)$ and $y=(b_1,b_2,\ldots, b_j,b_s)$ be two elements in $G.$
    Then  $\langle x,y\rangle=G$ if and only if $\langle a_i,b_i\rangle=S_{p_i}$ for all $1\leq i\leq j$ and $\langle a_s,b_s \rangle=\Z_s,$ where $s=p_{j+1}^{e_{j+1}}\dots p_k^{e_k}.$ 
\end{theorem}
\begin{proof}
    The argument follows analogously to the proof of the previous case.
\end{proof}
\begin{cor}
    Let $G=S_{p_1}\times  S_{p_2}\times  \dots\times  S_{p_j}\times  \Z_{s}.$  Then  the set of generating pairs of $G$ is given  $$\mathrm{Gen}(G)=\prod\limits_{i=1}^{j}\mathrm{Gen}(S_{p_i})\times \mathrm{Gen}(\Z_s).$$
\end{cor}
\begin{proof}
    This follows from the above theorem.
\end{proof}
Here $\mathrm{Gen}(\mathbb{Z}_s)$ denotes the set of generating pairs 
of $\mathbb{Z}_s$ for any $s\in \N$, defined precisely in the next section.

\section{Generating graphs of cyclic groups}\label{sec:4}
In this section, we review the structure of the generating graphs of 
finite cyclic groups, investigated thoroughly in~\cite{Me}, and 
simultaneously generalise the results to looped generating graphs. We 
adopt the notations from~\cite{Me}. We consider $\Z_n = \{0, 1, 2, 
\dots, (n-1)\}$, the group of integers modulo $n$. Let $n = p_1^{e_1} 
p_2^{e_2} \cdots p_r^{e_r}$ and $n_0 = p_1 p_2 \cdots p_r$, where 
$p_1 < p_2 < \cdots < p_r$ are distinct primes. 

\subsection{Looped generating graph of $\mathbb{Z}_n$}

We modify the definition of generating graph for $\Gamma(\Z_n)$ by allowing loops at 
vertices which generate $\Z_n$, and call it the \emph{looped generating 
graph}, denoted $\widehat{\Gamma}(\Z_n)$. Note that loops arise only 
for cyclic groups, since a loop at $x$ means $\langle x \rangle = G$, 
which forces $G$ to be cyclic. We adopt the convention that a loop 
contributes $1$ to the degree of its vertex.

We recall a few fundamental results from the paper~\cite{Me}, which will be needed later. A pair of elements $a \neq b \in V$ are adjacent in $\Gamma(\Z_n)$ if and only if 
$\langle a, b \rangle = \Z_n$, or equivalently, $\gcd(a, b, n) = 1$. 
Thus the set of generating pairs is given by
\begin{equation}\label{eq:class}
\mathrm{Gen}(\Z_n) = \{(a,b) \in \Z_n \times \Z_n \,:\,  \gcd(a,b) \in 
U(n)\},
\end{equation}
where $U(n)$ denotes the group of units of $\Z_n$. The equivalence 
relation on the vertex set $V$ of $\Gamma(\Z_n)$ is defined by 
$$a \sim b \text{ whenever } \gcd(a, n_0) = \gcd(b, n_0).$$ 
Corresponding to each divisor $d$ of $n_0$, the equivalence class is given by 
$$[d] = \{a \in \Z_n \,:\,  \gcd(a, n_0) = d\}.$$ 
Moreover, there are exactly $2^r$ such classes partitioning the vertex set $V$ of $\Gamma(\Z_n)$. The class $[1]$ 
induces a complete graph, and $[d_i]$ induces an empty graph for all 
$d_i \neq 1$. In addition, any two members of the same class have 
identical neighbourhoods in $\Gamma(\Z_n)$.

\begin{theorem}[Theorem~2.1,~\cite{Me}]
For a given $n$, the size of each equivalence class is 
$$|[d_i]| = \frac{n}{n_0}\varphi\!\left(\frac{n_0}{d_i}\right),$$
where $d_i \mid n_0$ and $1 \leq i \leq 2^r.$
\end{theorem}

\begin{theorem}[Theorem~3.6,~\cite{Me}]\label{thm:deg}
In $\Gamma(\Z_n)$, the vertex degree of $d_i \neq 1$ is given by 
$$\deg(d_i) = \frac{n}{d_i}\varphi(d_i).$$
\end{theorem}

We note that the edge set $\widehat{E}$ of $\widehat{\Gamma}(\Z_n)$,  includes all 
pairs $\{a,a\}$ for every $a \in U(n)$, in addition to the edges of 
$\Gamma(\Z_n)$, so $\mathrm{Gen}(\Z_n)$ extends accordingly. The 
equivalence classes and their sizes remain unchanged. Each generator 
$a \in [1]$ acquires a loop, contributing $1$ to its degree, giving 
$\deg(a) = (n-1) + 1 = n = \frac{n}{1}\varphi(1)$. Thus 
Theorem~\ref{thm:deg} extends uniformly to all $d_i \mid n_0$, 
including $d_i = 1$.

\begin{pro}
The number of edges in $\widehat{\Gamma}(\mathbb{Z}_n)$ is given by
$$|\widehat{E}| = \frac{\varphi(n)}{2}\biggl(\frac{n}{n_0}\sigma(n_0) 
- 1\biggr) + \varphi(n),$$
where $\sigma(n_0)$ denotes the sum of divisors of $n_0$.
\end{pro}
\begin{proof}
The first term equals $|E|$ of $\Gamma(\mathbb{Z}_n)$ from 
\cite[Corollary~3.2.4]{Me}, and the second term $\varphi(n)$ counts 
the loops added at each generator.
\end{proof}
% \begin{rem}
%    
% \end{rem}

\subsection{Adjacency and Laplacian Matrix}
For a graph \(\mathcal{G}\), let \(A(\G)\) and \(L(\G)\) denote its adjacency and Laplacian matrices, respectively. We now investigate the spectra of the adjacency and Laplacian matrices of \(\widehat{\Gamma}(\mathbb{Z}_n)\), which will be needed later in our study of abelian groups. Throughout, $J$ will denote the all-1 matrix, $\bf 0$ denote the all 0-matrix, and $I$ is the identity matrix.

\subsection*{Adjacency Matrix}
Based on our assumption, the adjacency matrix for the graph $\widehat\Gamma(\Z_n),$ denoted by $\widehat A,$ is a $0-1$ matrix and it is symmetric by the definition. The entries $a_{ij}$ are given by
\[a_{ij}=
\begin{dcases}
1&\quad\gcd(v_{i},v_{j})\in U(n);\\
0 &\quad\text{otherwise.}\\
\end{dcases}
\]
In the block form, $\widehat A=\Big[[d_{ij}]\Big],$ where $[d_{ij}]$ represents the block matrix corresponding to the classes $[d_{i}]$ and $[d_{j}],$ is given by 
        \[[d_{ij}]=
        \begin{dcases}
 J&\quad\gcd(d_{i},d_{j})=1;\\
 {\bf 0}&\quad\text{otherwise},\\
        \end{dcases}
        \]
where $J$ and $\bf 0$ are the matrices of order $\Big|[d_{i}]\Big|\times\Big|[d_{j}]\Big|.$ 
\begin{rem}
With respect to the equitable partition $\{d_1=1,d_2,\dots, d_{2^r}\},$ the quotient matrix of $\widehat A,$ denoted by $T^{\widehat A},$ is given by $T^A-D,$ where $T^A$ is the quotient matrix of the adjacency matrix $A$ of $\Gamma(\Z_n)$ and $D=\mathrm{diag}(-1,0,0,\dots, 0)$ (see Section 4.2,~\cite{Me}).
{$$T^A=\begin{pmatrix}
    \varphi(n)-1&a_1&a_2&\dots&\dots&a_{2^r-1}\\
    \varphi(n)&0&\theta_{23}&\dots&\dots&\theta_{2(2^r-1)}\\
    \varphi(n)&\theta_{32}&0&\theta_{34}&\dots&\theta_{3(2^r-1)}\\
    \vdots&\vdots&\vdots&\ddots&\vdots\\
    \varphi(n)&\theta_{2^r2}&\theta_{2^r3}&\dots&\dots&0
        \end{pmatrix},$$}
        where {\[ \theta_{ij}=\begin{dcases}
        a_{j-1}&\quad \text{ if } \gcd(d_i,d_j)=1;\\
        0 &\quad \text{ if } \gcd(d_i,d_j)\neq 1.
    \end{dcases}\]}
   and {$a_i=\Big|[d_{i+1}]\Big|$} for $1\leq i\leq 2^r-1.$

Further, we note that in \cite[Proposition 4.2.2]{Me}, the matrix $T^A$ is similar to $Q_n=\widetilde Q_n+D,$ and $\widetilde Q_n$ is given by 
$$\widetilde Q_n=\begin{pmatrix}
    \varphi(n)&\rho_{1,2}&\dots &\rho_{1,2^r-1}&\rho_{1,2^r}\\
    \rho_{1,2}&0&\dots&\rho_{2,2^r-1}&\rho_{2,2^r}\\
    \vdots&\vdots&\ddots&\vdots&\vdots\\
    \rho_{1,2^r}&\rho_{2,2^r}&\dots&\rho_{2^r-1,2^r}&0
\end{pmatrix}.$$
For $1\leq \ell\leq 2^r-1,$ $a_\ell=\frac{n}{n_0}\varphi\left(\frac{n_0}{d_{\ell+1}}\right)$ and thus $$\rho_{i,j}=\begin{dcases}
    \sqrt{a_{i-1}a_{j-1}}=\frac{n}{n_0}\sqrt{\varphi\left(\frac{n_0}{d_i}\right)\varphi\left(\frac{n_0}{d_j}\right)} & \text{ if } \gcd(d_i,d_j)=1;\\
    \sqrt{\varphi(n)a_{j-1}}=\frac{n}{n_0}\sqrt{\varphi(n_0)\varphi\left(\frac{n_0}{d_j}\right)}& \text{ if } d_i=1 \text{ and } d_j\neq 1;\\
    \quad 0& \text{otherwise}.
\end{dcases}$$
\end{rem}

From the above discussion, we have the following result.
\begin{pro}
    The quotient matrix $T^{\widehat A}$ is similar to $\widetilde Q_n.$
\end{pro}
\begin{proof}
Since $T^{\widehat A}=T^A-D$ and $T^A=P^{-1}Q_nP,$
where\[P=\operatorname{diag}\!\left(\sqrt{\varphi(n)},\sqrt{a_1},\dots,\sqrt{a_{2^r-1}}\right),\] it follows, as \(P\) and \(D\) are diagonal matrices, that \(T^{\widehat A}\) is similar to \(Q_n-D=\widetilde Q_n\).
\end{proof}

The matrix $\widetilde Q_n$ is a well structured matrix of size $2^r\times 2^r$.
\begin{theorem}[Proposition 4.8, Theorem 4.9,~\cite{Me}]
    The matrix $\widetilde Q_{n}$ is given by 
    \begin{equation*}
        \widetilde Q_{n}=\frac{n}{n_0}\widetilde Q_{n_0},
       \end{equation*}
where $\widetilde Q_{n_0}=\sqrt {\varphi(n_0)}\Bigg(U_1\otimes U_2\otimes \dots \otimes U_r\Bigg)$ and for $ 1\leq j\leq r,$ 
$U_i=\begin{pmatrix} \sqrt{\varphi(p_{r-{(i-1)}})}&1\\1&0
\end{pmatrix}.$
\end{theorem}
Then the properties of Kronecker product of matrices gives the following result.
\begin{cor}[Corollary 4.10,~\cite{Me}]\label{cor:Etensor}
    The eigenvalues of $\widetilde{Q}_{n_0}$ are of the form
    $$\sqrt{\varphi(n_0)}\prod_{i=1}^{r} u_{i,\ell_i},$$
    parametrized over all choices $(\ell_1,\ldots,\ell_r)\in\{1,2\}^r$, 
    where for each $i$,
    $$u_{i,1} = \frac{\sqrt{\varphi(p_{r+1-i})} + \sqrt{4 + 
    \varphi(p_{r+1-i})}}{2}, \qquad u_{i,2} = \frac{\sqrt{\varphi
    (p_{r+1-i})} - \sqrt{4 + \varphi(p_{r+1-i})}}{2}$$
    are the two eigenvalues of the $i$-th factor matrix $U_i$.
    \end{cor}
From the above discussion we conclude the follwing result for $\widehat\Gamma(\Z_n)$.
\begin{theorem}\label{thm:Aloop}
The characteristic polynomial of the matrix $\widehat A$ of $\widehat \Gamma(\Z_n)$ is given by  
$$\phi_{\widehat A}(x)=\frac{n}{n_0}\phi_{\widetilde Q_{n_0}}(x)\,{x^{n-2^r}},$$
where $\phi_{\widetilde Q_{n_0}}(x)$ is the characteristic polynomial of $\widetilde Q_{n_0}$ and $r$ denotes the number of distinct prime divisors of $n.$
\end{theorem}
\begin{proof}
Note that the quotient matrix $T^{\widehat A}$ is similar to $\widetilde Q_n$ and and the subgraph induced by each equivalence class $[d_i]$ is either empty or complete, it follows that $\mathrm{rank}(\widehat A)=\mathrm{rank}(\widetilde Q_n).$ Hence the result follows from Theorem~\ref{thm:specA}.
\end{proof}
\subsection*{Laplacian Matrix}

The Laplacian matrix of a graph $\G = (\V, \E)$ on $n$ vertices $\{v_1,\dots ,v_n\}$ is 
defined as $L(\G) = D(\G) - A(\G)$, where $D(\G) = \mathrm{diag}
(\deg(v_1), \dots, \deg(v_n))$. It is well known that $L(\G)$ is symmetric 
and positive semidefinite with smallest eigenvalue $0$.

We denote by $L$ and $\widehat{L}$ the Laplacian matrices of 
$\Gamma(\mathbb{Z}_n)$ and $\widehat{\Gamma}(\mathbb{Z}_n),$ 
respectively. In block form, indexed by the equivalence classes 
$[d_1], \dots, [d_{2^r}]$, the matrix $L=\left[[l_{ij}]\right]$ of $\Gamma(\mathbb{Z}_n)$ is given by
\[
[l_{ij}] = \begin{dcases}
-J       & \gcd(d_i, d_j) = 1, \ d_i \neq d_j; \\
nI - J   & d_i = d_j = 1; \\
\deg(d_i)I & d_i = d_j \neq 1; \\
\bf{0}        & \text{otherwise,}
\end{dcases}
\]
where $J$ and $\bf 0$ are matrices of order $|[d_i]| \times |[d_j]|$. 
We refer the reader to~\cite{Me} for the detailed study of the 
Laplacian matrix of $\Gamma(\mathbb{Z}_n)$.

\begin{rem}
Adding loops at generators affects only the block $[l_{11}]$ of $L$, 
corresponding to the class $[1]$. In $\widehat{\Gamma}(\mathbb{Z}_n)$, 
each $a \in [1]$ has degree $n$ (the original $n-1$ neighbours plus 
the loop), so
$$[\widehat{l}_{11}] = nI - J.$$
In $\Gamma(\mathbb{Z}_n)$, the same block is $(n-1)I - (J - I) = nI - J$.
Hence $[\widehat{l}_{11}] = [l_{11}]$, and all other blocks are 
unaffected. Therefore $\widehat{L} = L$, and all results established 
for $L$ carry over to $\widehat{L}$ without change.
\end{rem}

\section{The generating graph of abelian groups}\label{sec:5}
In this section we investigate the structure of the generating graph $\Gamma(G)$ for non-cyclic abelian groups $G$. 
\subsection{Graph Structure}
We begin with the $p$-group case, which will serves as the fundamental building block for the general case.

Let $\Delta(G)$ denote the subgraph of $\Gamma(G)$ induced by its non-isolated vertices, and let $\mathcal{I}(G)$ 
denote the set of isolated vertices of $\Gamma(G)$. We adopt the notation $K_{\ell\times s}$ for the complete $\ell$-partite graph $K_{\underbrace{s,s,\ldots,s}_{\ell\text{ times}}}.$ These notations we will use throughout.

\begin{theorem}\label{thm:pgraph}
    Let $G=\Z_{p^e}\times  \Z_{p^t},$ where $p$ is a prime. Then 
    \begin{enumerate}[label=(\roman*)]
        \item the set $\mathcal{I}(G)=\Phi(G);$
        \item  the graph $\Delta(G)$ is a {$r$-regular} graph, where $r=p\varphi(p^{e+t-1});$ 
        \item the graph $\Delta(G)$ is a complete $(p+1)$ partite graph, isomorphic to $K_{(p+1)\times \varphi(p^{e+t-1})}.$ Moreover, $|E(\Delta(G))|=\frac{1}{2}p(p+1)(\varphi(p^{e+t-1}))^2.$
       \end{enumerate}
\end{theorem}
\begin{proof}
    These follow from Proposition~\ref{pro:pgen} and Remark~\ref{rem:edge}.
\end{proof}

Now let us investigate the structure for general case.
\begin{pro}\label{pro:iso1}
        Let $G = S_{p_1} \times \cdots \times S_{p_k}$. An element 
        $(a_1, a_2 \ldots, a_k) \in G$ is an isolated vertex of 
        $\Gamma(G)$ if and only if $a_i \in \mathcal{I}(S_{p_i})$ 
        for at least one $1 \leq i \leq k$. Explicitly,
        $$\I(G) = \bigcup_{i=1}^{k} S_{p_1} \times \cdots \times 
        \I(S_{p_i}) \times \cdots \times S_{p_k}.$$
\end{pro}
\begin{proof}
    The statement follows directly from Theorem~\ref{thm:agen}.
\end{proof}

\begin{theorem}\label{thm:case1}
    Let $G= S_{p_1}\times  \ldots \times S_{p_k},$ where each $S_{p_i}$'s are non-cyclic and $|S_{p_i}|=p_i^{u_i}$ Then  $$\Ga(G)=\otimes_{i=1}^{k}\Ga(S_{p_i}).$$ Moreover,
    $\Delta(G)$ is a $\alpha$-partite graph, where $\alpha=\prod\limits_{i=1}^{k}(p_i+1)$ and each part has size $\prod\limits_{i=1}^{k}\varphi(p_i^{u_i-1}).$
\end{theorem}
\begin{proof}
    From Theorem~\ref{thm:agen}, Theorem~\ref{thm:pgraph} and by the definition of the Kronecker product of graphs, the statement follows.
\end{proof}
\begin{cor}\label{cor:case1}
    Let $G = S_{p_1} \times \cdots \times S_{p_k}$. Then $\Delta(G)$ is 
    $d$-regular, where
    \[
    d = \prod_{i=1}^{k} r_i,
    \]
    and $r_i$ denotes the degree of $\Delta(S_{p_i})$ for each $1 \leq i \leq k$.
\end{cor}
\begin{proof}
The result follows from Theorem~\ref{thm:case1} and the properties 
of the Kronecker product of graphs (see Proposition~\ref{pro:per}).
\end{proof}

\begin{pro}\label{pro:iso2}
   For $k\neq j.$ Let $G=S_{p_1}\times  \ldots \times  S_{p_j}\times  \Z_s,$ where $s=\prod \limits_{i=j+1}^{k}p_i^{e_i}.$  A non-zero element $(a_1,a_2,\ldots,a_j,a_s)\in G$ is an isolated vertex in $\Gamma(G)$ if and only if  $a_i$ is an isolated vertex of $\Gamma(S_{p_i})$ for some $i\in \{1,2,\ldots ,j\}.$ Explicitly, 
   $$\I(G) = \bigcup_{i=1}^{k} S_{p_1} \times \cdots \times 
   \I(S_{p_i}) \times \cdots \times S_{p_j}\times \Z_s.$$
\end{pro}
\begin{proof}
    This follows from Theorem~\ref{thm:bgen}. Also, the graph $\widehat\Ga(\Z_s)$ doesn't have any isolated vertex, thus there is no isolated vertex in $\Ga(G)$ of the form $(a_1,a_2,\dots, a_j,a_s)$ such that for each  $1\leq i\leq j,$ $ a_i\in S_{p_i}\setminus \I(S_{p_i}),$ and $a_s\in\Z_s.$
\end{proof}

\begin{theorem}\label{thm:case2}
    Let $G=S_{p_1}\times  \ldots \times  S_{p_j}\times  \Z_s,$ where $s=\prod \limits_{i=j+1}^{k}p_i^{e_i}.$ Then $$\Ga(G)=\left(\otimes_{i=1}^{j}\Ga(S_{p_i})\right)\otimes\,\widehat\Ga(\Z_s).$$ 
\end{theorem}
\begin{proof}
    This follows from the definition of the Kronecker product of graphs and Theorem~\ref{thm:bgen}.
\end{proof}

\begin{cor}\label{cor:degz}
    Let $G=S_{p_1}\times  \ldots \times  S_{p_j}\times  \Z_s.$ Then $\Delta(G)$ is not regular. Let $r_i's$ be the degrees of the respective regular graphs $\Delta(S_{p_i}),$ where $1\leq i \leq j. $ Then for each $a=(a_1,\dots a_j,a_s)\in V\setminus\I(G),$ the vertex degree is given by
    $$\deg(a)=r_1r_2\ldots r_j\deg(a_s),$$ where $\deg(a_s)$ is the vertex degree of $a_s$  in $\widehat \Ga(\Z_s).$
\end{cor}
\begin{proof}
    Note that the regularity of $\Delta(G)$ depends on the vertex degree of the last component $a_s\in \Z_s$ in $\widehat\Ga(\Z_s)$ of each $a\in G.$ However, the graph $\widehat \Ga(\Z_s)$ is not regular since $\deg 0\neq \deg 1.$ Then $\Delta(G)$ is not regular. Then the result follows from the properties of Kronecker product of graphs (see Proposition~\ref{pro:per}) and Theorem~\ref{thm:pgraph}.
\end{proof}
\begin{rem}\label{rem:decompose}
        Let $G = S_{p_1}\times\cdots\times S_{p_k},$ where each $S_{p_i}$ is 
        a non-cyclic abelian $p_i$-group. Since $(a_1,\ldots,a_k)$ is 
        non-isolated in $\Gamma(G)$ if and only if $a_i$ is non-isolated in 
        $\Gamma(S_{p_i})$ for each $i$, we have
        $$\Gamma(G) = \otimes_{i=1}^{k}\Gamma(S_{p_i}), \qquad 
        \Delta(G) = \otimes_{i=1}^{k}\Delta(S_{p_i}).$$
        More generally, for $G = S_{p_1}\times\cdots\times S_{p_j}\times 
        \mathbb{Z}_s$ with $s = \prod\limits_{i=j+1}^{k}p_i^{u_i}$, since 
        $\widehat{\Gamma}(\mathbb{Z}_s)$ has no isolated vertices,
        $$\Gamma(G) = \left(\otimes_{i=1}^{j}\Gamma(S_{p_i})\right)
        \otimes\,\widehat{\Gamma}(\mathbb{Z}_s), \qquad \Delta(G) = 
        \left(\otimes_{i=1}^{j}\Delta(S_{p_i})\right)\otimes\,
        \widehat{\Gamma}(\mathbb{Z}_s).$$
\end{rem}

Next, we illustrate Theorem~\ref{thm:case1} and Theorem~\ref{thm:case2} with two examples: one where all 
Sylow subgroups are non-cyclic, and one where some are cyclic.
\begin{ex}[Case1]
    Let $G = (\mathbb{Z}_2 \times \mathbb{Z}_2) \times (\mathbb{Z}_3 \times 
    \mathbb{Z}_3)$. The Sylow $2$-subgroup is $S_2 = \mathbb{Z}_2 \times 
    \mathbb{Z}_2$ and the Sylow $3$-subgroup is $S_3 = \mathbb{Z}_3 \times 
    \mathbb{Z}_3$. The generating graphs $\Gamma(S_2)$, $\Gamma(S_3)$, and 
    their Kronecker product $\Gamma(G) = \Gamma(S_2) \otimes \Gamma(S_3)$ 
    are illustrated below:
\end{ex}

\begin{figure}[H]
    \centering
\begin{minipage}[t]{0.45\textwidth}
    \centering
\begin{tikzpicture}[every node/.style={circle, draw, fill=blue, inner sep=2.1pt}, scale=1.5]

    % Nodes for K3
    \node (A) at (0,0.5) [label=below left:{(0,1)}] {};
    \node (B) at (1,0.5) [label=below right:{(1,0)}] {};
    \node (C) at (0.5,1) [label=above:{(1,1)}] {};
    \node (D) at (-0.6,1) [label=above:{(0,0)}] {};
    % Edges for K3
    \draw (A) -- (B);
    \draw (B) -- (C);
    \draw (C) -- (A);
     \end{tikzpicture}
     \captionof{figure}{Graph $\Gamma(S_2)$}
\end{minipage}
\hfill
\begin{minipage}[t]{0.45\textwidth}
        \centering
    \begin{tikzpicture}[every node/.style={circle, fill=blue, inner sep=2.1pt}, scale=1]
      % Nodes
      \node (a1) at (-0.5, 2) [label=above:{(0,1)}] {};
      \node (a2) at (0.5, 2)  [label=above:{(0,2)}] {};
    
      \node (b1) at (-0.5, -2) [label=left:{(1,0)}] {};
      \node (b2) at (0.5, -2)  [label=right:{(2,0)}] {};
    
      \node (c1) at (-2, 0.5) [label=left:{(1,1)}] {};
      \node (c2) at (-2, -0.5) [label=left:{(2,2)}] {};
    
      \node (d1) at (2, 0.5) [label=right:{(1,2)}] {};
      \node (d2) at (2, -0.5) [label=right:{(2,1)}] {};
      \node (e) at (-2,2) [label=above:{(0,0)}]{};
    
      % Draw edges between vertices from different parts
      \foreach \u in {a1,a2}
        \foreach \v in {b1,b2,c1,c2,d1,d2}
          \draw (\u) -- (\v);
    
      \foreach \u in {b1,b2}
        \foreach \v in {c1,c2,d1,d2}
          \draw (\u) -- (\v);
    
      \foreach \u in {c1,c2}
        \foreach \v in {d1,d2}
          \draw (\u) -- (\v);
     \end{tikzpicture}
    \captionof{figure}{Graph $\Gamma(S_3)$}
\end{minipage}
\end{figure}
\begin{figure}[H]
    \centering
    \begin{tikzpicture}[
        every node/.style={circle, fill=blue, inner sep=2pt, font=\small},
        x=1.2cm, y=1.2cm
    ]
% === Layer 0 ===
\node (a10) at (0,8) [label=right:{(0,1,0,1)}] {};
\node (a20) at (0.5,7) [label=right:{(0,1,0,2)}] {};
\node (b10) at (1,6) [label=right:{(0,1,1,0)}] {};
\node (b20) at (1.5,5) [label=right:{(0,1,2,0)}] {};
\node (c10) at (2,4) [label=right:{(0,1,1,2)}] {};
\node (c20) at (2.5,3) [label=right:{(0,1,2,1)}] {};
\node (d10) at (3,2) [label=right:{(0,1,1,1)}] {};
\node (d20) at (3.5,1) [label=right:{(0,1,2,2)}] {};

% % === Layer 1 ===
\node (a11) at (-1.5,8) [label=left:{(1,0,0,1)}] {};
\node (a21) at (-2,7) [label=left:{(1,0,0,2)}] {};
\node (b11) at (-2.5,6) [label=left:{(1,0,1,0)}] {};
\node (b21) at (-3,5) [label=left:{(1,0,2,0)}] {};
\node (c11) at (-3.5,4) [label=left:{(1,0,1,2)}] {};
\node (c21) at (-4,3) [label=left:{(1,0,2,1)}] {};
\node (d11) at (-4.5,2) [label=left:{(1,0,1,1)}] {};
\node (d21) at (-5,1) [label=left:{(1,0,2,2)}] {};

% % === Layer 2 ===
\node (a12) at (-4.5,0) [label=left:{(1,1,0,1)}] {};
\node (a22) at (-3.8,0) [label=below:{(1,1,0,2)}] {};
\node (b12) at (-2.5,0) [label=below:{(1,1,1,0)}] {};
\node (b22) at (-1.3,0) [label=below:{(1,1,2,0)}] {};
\node (c12) at (0,0) [label=below:{(1,1,1,2)}] {};
\node (c22) at (1.2,0) [label=below:{(1,1,2,1)}] {};
\node (d12) at (2.4,0) [label=below:{(1,1,1,1)}] {};
\node (d22) at (3.2,0) [label=right:{(1,1,2,2)}] {};
%Isolated vertices
\node (a) at (-6.5,-1.4) [label=below:{(0,0,0,1)}] {};
\node (b) at (-5.3,-1.4) [label=below:{(0,0,0,2)}] {};
\node (c) at (-4.1,-1.4) [label=below:{(0,0,1,0)}] {};
\node (d) at (-2.8,-1.4) [label=below:{(0,0,2,0)}] {};
\node (e) at (-1.5,-1.4) [label=below:{(0,0,1,2)}] {};
\node (f) at (-0.2,-1.4) [label=below:{(0,0,2,1)}] {};
\node (g) at (1.0,-1.4) [label=below:{(0,0,1,1)}] {};
\node (h) at (2.2,-1.4) [label=below:{(0,0,2,2)}] {};
\node (i) at (3.4,-1.4) [label=below:{(0,1,0,0)}] {};
\node (j) at (4.6,-1.4) [label=below:{(1,0,0,0)}] {};
\node (k) at (5.8,-1.4) [label=below:{(1,1,0,0)}] {};
\node (l) at (7.0,-1.4) [label=below:{(0,0,0,0)}] {};

%%Edges***************************
\foreach \u in {a10,a20}
\foreach \v in {b11,b21,c11,c21,d11,d21,b12,b22,c12,c22,d12,d22}
  \draw (\u) -- (\v);

\foreach \u in {b10,b20}
\foreach \v in {a11,a21,c11,c21,d11,d21,a12,a22,c12,c22,d12,d22}
  \draw (\u) -- (\v);

\foreach \u in {c10,c20}
\foreach \v in {a11,a21,b11,b21,d11,d21,a12,a22,b12,b22,d12,d22}
  \draw (\u) -- (\v);
  \foreach \u in {d10,d20}
  \foreach \v in {a11,a21,b11,b21,c11,c21,a12,a22,b12,b22,c12,c22}
    \draw (\u) -- (\v);

\foreach \u in {a11,a21}
\foreach \v in {b12,b22,c12,c22,d12,d22}
\draw (\u) -- (\v);

\foreach \u in {b11,b21}
\foreach \v in {a12,a22,c12,c22,d12,d22}
\draw (\u) -- (\v);  

\foreach \u in {c11,c21}
\foreach \v in {a12,a22,b12,b22,d12,d22}
\draw (\u) -- (\v);  

\foreach \u in {d11,d21}
\foreach \v in {a12,a22,b12,b22,c12,c22}
\draw (\u) -- (\v);  
\end{tikzpicture}
\captionof{figure}{$\Gamma(G) = \Gamma(S_2) \otimes \Gamma(S_3)$}
\end{figure}
Note that $\Gamma(G)$ is a 12-regular graph and the isolated vertices are exactly the elements of $(\Phi(S_2) \times S_3 )\cup (S_2 \times \Phi(S_3))$.

\begin{ex}[Case2]
    Now we consider $G = (\mathbb{Z}_4 \times \mathbb{Z}_2) \times \mathbb{Z}_3$, 
    an example where the Sylow subgroups are not all of the same type: the 
    Sylow $2$-subgroup $S_2 = \mathbb{Z}_4 \times \mathbb{Z}_2$ is non-cyclic, 
    while the Sylow $3$-subgroup $S_3 = \mathbb{Z}_3$ is cyclic.
    The generating 
    graphs $\Gamma(S_2)$, $\Gamma(S_3)$, and their Kronecker product 
    $\Gamma(G) = \Gamma(S_2) \otimes \widehat\Gamma(S_3)$ are illustrated below:
\end{ex}
\begin{figure}[ht!]
   \begin{minipage}[t]{0.45\textwidth}
        \centering
    \begin{tikzpicture}[every node/.style={circle, fill=blue, inner sep=1.5pt}]
      % Nodes
      \node (a1) at (-0.5, 2) [label=above:{(0,1)}] {};
      \node (a2) at (0.5, 2)  [label=above:{(2,1)}] {};
    
      \node (d1) at (-2, 1.5) [label=left:{(0,0)}] {};
      \node (d2) at (2, 1.5)  [label=right:{(2,0)}] {};
    
      \node (c1) at (-2, 0.5) [label=left:{(1,0)}] {};
      \node (c2) at (-2, -0.5) [label=left:{(3,0)}] {};
    
      \node (b1) at (2, 0.5) [label=right:{(1,1)}] {};
      \node (b2) at (2, -0.5) [label=right:{(3,1)}] {};

      % Draw edges between vertices from different parts
      \foreach \u in {a1,a2}
        \foreach \v in {c1,c2,b1,b2}
          \draw (\u) -- (\v);
    
      \foreach \u in {c1,c2}
        \foreach \v in {b1,b2}
          \draw (\u) -- (\v);
    
    \end{tikzpicture}
\captionof{figure}{Graph $\Ga(S_2)$}
\end{minipage}
\hfill
    \begin{minipage}[t]{0.45\textwidth}
    \centering
\begin{tikzpicture}[every node/.style={circle, draw, fill=blue, inner sep=2.1pt}, scale=2.8]

    % Nodes for K3
    \node (A) at (0,0.5) [label=below left:{1}] {};
    \node (B) at (1,0.5) [label=below right:{2}] {};
    \node (C) at (0.5,1) [label=above:{0}] {};
    % Edges for K3
    \draw (A) -- (B);
    \draw (B) -- (C);
    \draw (C) -- (A);
    \draw (A) to[loop left] ();
    \draw (B) to[loop right] ();
     \end{tikzpicture}
     \captionof{figure}{Graph $\widehat \Ga(S_3)$}
\end{minipage}
\end{figure}

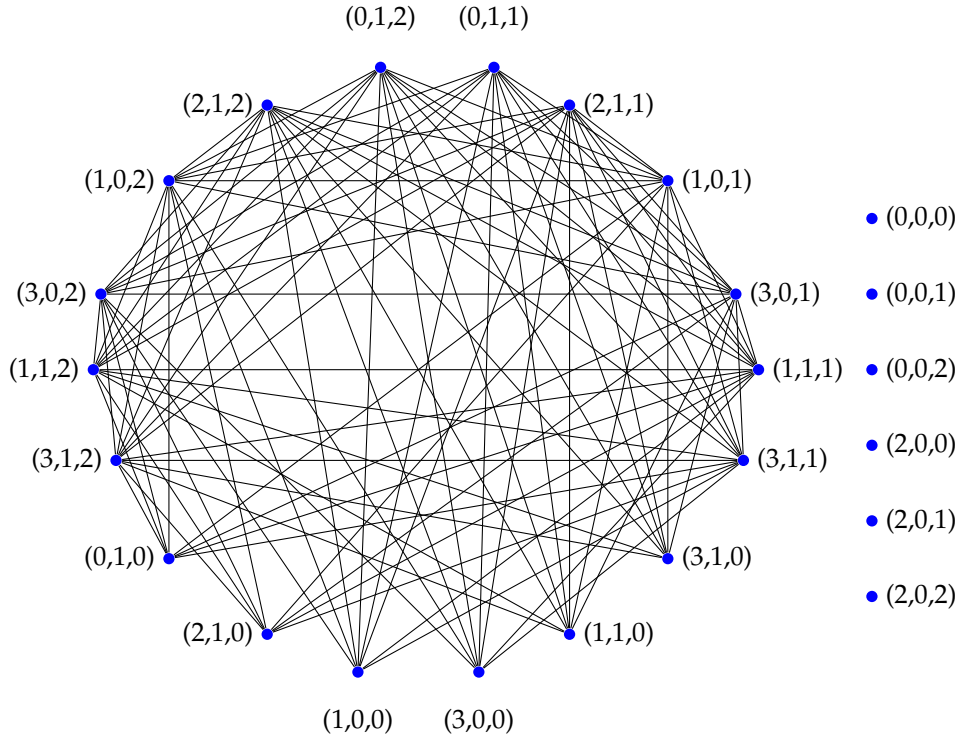
\begin{figure}[h!]
    \centering
    \begin{tikzpicture}[
        every node/.style={circle, fill=blue, inner sep=1.5pt, font=\small},
        x=1cm, y=1cm
    ]
% === Layer 0 ===
\node (a10) at (0.5,8) [label=above:{(0,1,1)}] {};
\node (a20) at (1.5,7.5) [label=right:{(2,1,1)}] {};
\node (b10) at (2.8,6.5) [label=right:{(1,0,1)}] {};
\node (b20) at (3.7,5) [label=right:{(3,0,1)}] {};
\node (c10) at (4,4) [label=right:{(1,1,1)}] {};
\node (c20) at (3.8,2.8) [label=right:{(3,1,1)}] {};

% % === Layer 1 ===
\node (a11) at (-1,8) [label=above:{(0,1,2)}] {};
\node (a21) at (-2.5,7.5) [label=left:{(2,1,2)}] {};
\node (c11) at (-3.8,6.5) [label=left:{(1,0,2)}] {};
\node (c21) at (-4.7,5) [label=left:{(3,0,2)}] {};
\node (b11) at (-4.8,4) [label=left:{(1,1,2)}] {};
\node (b21) at (-4.5,2.8) [label=left:{(3,1,2)}] {};

% % === Layer 2 ===
\node (a12) at (-3.8,1.5) [label=left:{(0,1,0)}] {};
\node (a22) at (-2.5,0.5) [label=left:{(2,1,0)}] {};
\node (b12) at (-1.3,0) [label=below:{(1,0,0)}] {};
\node (b22) at (0.3,0) [label=below:{(3,0,0)}] {};
\node (c12) at (1.5,0.5) [label=right:{(1,1,0)}] {};
\node (c22) at (2.8,1.5) [label=right:{(3,1,0)}] {};
%Isolated vertices
\node (a) at (5.5,6) [label=right:{(0,0,0)}] {};
\node (b) at (5.5,5) [label=right:{(0,0,1)}] {};
\node (c) at (5.5,4) [label=right:{(0,0,2)}] {};
\node (d) at (5.5,3) [label=right:{(2,0,0)}] {};
\node (e) at (5.5,2) [label=right:{(2,0,1)}] {};
\node (f) at (5.5,1) [label=right:{(2,0,2)}] {};
%%Edges***************************
\foreach \u in {a10,a20}
\foreach \v in {b11,b21,c11,c21,b12,b22,c12,c22}
  \draw (\u) -- (\v);

\foreach \u in {b10,b20}
\foreach \v in {a11,a21,c11,c21,a12,a22,c12,c22}
  \draw (\u) -- (\v);

\foreach \u in {c10,c20}
\foreach \v in {a11,a21,b11,b21,a12,a22,b12,b22}
  \draw (\u) -- (\v);

\foreach \u in {a11,a21}
\foreach \v in {b12,b22,c12,c22}
\draw (\u) -- (\v);

\foreach \u in {b11,b21}
\foreach \v in {a12,a22,c12,c22}
\draw (\u) -- (\v);  

\foreach \u in {c11,c21}
\foreach \v in {a12,a22,b12,b22}
\draw (\u) -- (\v);   
%%because of loops..

\foreach \u in {a10,a20}
\foreach \v in {b10,b20,c10,c20}
\draw (\u) -- (\v); 
\foreach \u in {b10,b20}
\foreach \v in {c10,c20}
\draw (\u) -- (\v); 
\foreach \u in {a11,a21}
\foreach \v in {b11,b21,c11,c21}
\draw (\u) -- (\v); 
\foreach \u in {b11,b21}
\foreach \v in {c11,c21}
\draw (\u) -- (\v); 
\end{tikzpicture}
\captionof{figure}{$\Gamma(G) = \Gamma(S_2) \otimes \widehat\Gamma(S_3)$}
\end{figure}

\newpage
Note that $\Gamma(G)$ is non-regular: vertices in $S_2 \setminus \Phi(S_2)$ 
paired with the non-identity elements of $S_3$ have degree $12$, those paired 
with the identity have degree $8$, and the isolated vertices (arising solely 
from $\Phi(S_2) \times S_3$) have degree $0$. 

We record the following consequences of the preceding analysis on the behaviour of generating pairs of the graph.
\subsection{Structural Characterization}
We now derive some structural characterizations of the graph 
$\Gamma(G)$ for finite abelian groups $G$ with $d(G) = 2$. Since any 
such group decomposes as a direct product of its Sylow subgroups with 
pairwise coprime orders, and we saw that the structure of $\Gamma(G)$ is also completely 
determined by the generating graphs of these Sylow subgroups via the 
Kronecker product. Particularly in this section we establish necessary and sufficient 
conditions for $\Gamma(G)$ to be regular, determine when the isolated 
vertices form a subgroup, and characterize when two such graphs are 
isomorphic.
\begin{theorem}[Regularity]
    Let \(G\) be a finite abelian group with $d(G)=2$.  Then the graph \(\Delta(G)\) is regular if and only if all Sylow subgroups of $G$ are non-cyclic.
    \end{theorem}
    \begin{proof}
    Suppose $G\cong S_{p_1}\times \dots \times S_{p_k},$ where each Sylow component $S_{p_i}$ is non-cyclic. Then by Theorem~\ref{thm:pgraph} (for $p$ groups) and Corollary \ref{cor:case1} (for k=j), $\Delta(G)$ is a regular graph. Now to prove the necessary part, we suppose that $\Delta(G)$ is a regular graph for an abelian group $G.$ Assume that $G$ has a cyclic Sylow $p$-subgroup. Then $G\cong S_{p_1}\times S_{p_2}\times \dots S_{p_j}\times \Z_s$ such that $p_i\nmid s$ for all $i.$ Let $x=(a_1,\dots a_j,0)$ and $y=(a_1,\dots,a_j,1)\in G.$ Since $\deg_{\widehat \Ga(\Z_s)}(0)\neq \deg_{\widehat\Ga(\Z_s)} (1),$ thus by Corollary~\ref{cor:degz}, we get $\deg_{\Delta(G)}(x)\neq \deg_{\Delta(G)}(y)$ which contradicts that $\Delta(G)$ is regular. Thus $\Delta(G)$ is regular if and only if every Sylow subgroup of $G$ 
    is non-cyclic.
    \end{proof}
    \begin{theorem}[Isolated Vertices]\label{thm:subgroup}
        Let $G$ be a finite non-cyclic abelian group with $d(G) = 2$. Then 
        $\I(G)$ is a subgroup of $G$ if and only if $G$ is a $p$-group, in  particular $\I(G) = \Phi(G)$.
    \end{theorem}
    \begin{proof}
      Suppose $G$ is not a $p$-group. Then $G$ has at least 
        two distinct prime divisors, and by the primary decomposition 
        $G = S_{p_1} \times \cdots \times S_{p_k}$ with $k \geq 2$. An 
        element $(a_1,\ldots,a_k) \in G$ is isolated if and only if 
        $a_i \in \I(S_{p_i})$ for at least one $i$. Hence 
        $$\I(G) = \bigcup_{i=1}^{k} S_{p_1} \times \cdots \times 
        \I(S_{p_i}) \times \cdots \times S_{p_k}.$$
        Taking $x = (0,\ldots, a_i, \ldots, 0)$ and 
        $y = (0,\ldots, b_j,\ldots, 0)$ with $a_i \in \I(S_{p_i})$, 
        $b_j \in \I(S_{p_j})$ and $i \neq j$, we have $x, y \in \I(G)$ 
        but $x + y \notin \I(G)$. So $\I(G)$ is not closed, hence not 
        a subgroup.
    
    Conversely, suppose $G$ is a non-cyclic abelian $p$-group. An element $x \in G\setminus\{e\}$ is isolated if 
        and only if $x \in \Phi(G)$ (see Theorem~\ref{thm:pgraph}).  Hence $\I(G) = \Phi(G)$, which is a subgroup.
    \end{proof}   

\textbf{Note:} For a group $G,$ we define a set $\I_G:=\{x\in G\,:\,\langle x,y\rangle\neq G \text{ for all } y\in G\}.$ Note that the set $\I_G$ of $G$ corresponds to the set $\I(G)$ in $\Gamma(G).$
\begin{lemma}
        Let $A$ and $B$ be finite 2-generated groups with $\gcd(|A|, |B|) = 1$. 
        Then $G=A \times B$ is 2-generated and
        $$|G| - |{\I}_G| = \bigl(|A| - |{\I}_A|\bigr)
        \bigl(|B| - |{\I}_B|\bigr).$$     
\end{lemma}
\begin{proof}
Note that $\langle (a,b),(a',b')\rangle = G$ if and only if $\langle a,a'\rangle = A$ 
and $\langle b,b'\rangle = B$. Thus $(a,b) \notin {\I}_G$ if and 
only if $a \notin {\I}_A$ and $b \notin {\I}_B$, and counting 
such pairs gives the formula.
\end{proof}

    \begin{theorem}[Graph Isomorphism]\label{thm:isom}
        Let $G$ and $H$ be non-isomorphic finite abelian groups of the same order 
        $n = p_1^{u_1} \cdots p_k^{u_k}$, each with minimum generating set of 
        size $2$. Then $\Gamma(G) \cong \Gamma(H)$ if and only if, for each prime 
        $p_i \mid n$, the Sylow $p_i$-subgroups $S_{p_i}(G)$ and $S_{p_i}(H)$ 
        are either both cyclic or both non-cyclic.
        \end{theorem}

\begin{proof}
Since the Sylow subgroups have coprime orders, over all Sylow subgroups, by setting 
$c_i = |S_{p_i}| - |{\I}_{S_{p_i}}|,$ from the above lemma, we get
$$|{\I}_G| = |G| - \prod_{i=1}^{k} c_i.$$

From Theorem~\ref{thm:pgraph}, the values of $c_i$ are
$$c_i = \begin{cases} p_i^{u_i}
& \text{if } S_{p_i} \text{ is cyclic,}\\[4pt] 
p_i^{u_i-2}(p_i^2-1) & \text{if } S_{p_i} \text{ is non-cyclic.}
\end{cases}$$

Note that $c_i$ depends only on $p_i$, $e_i$ and whether $S_{p_i}$ is 
cyclic or not, not on the internal structure of a non-cyclic Sylow 
subgroup.

To prove the converse, we suppose $S_{p_i}(G)$ and $S_{p_i}(H)$ have the 
same cyclic (non-cyclic) type at every prime $p_i$. Then $c_i(G) = c_i(H)$ 
for all $i$, so $|\I_G| = |\I_H|$. Moreover, by Theorem~\ref{thm:pgraph}, 
$\Delta(S_{p_i}(G)) \cong \Delta(S_{p_i}(H))$ at each prime (both are 
$K_{(p_i+1)\times \varphi(p_i^{u_i-1})}$ in the non-cyclic case, and 
identical in the cyclic case since the structure depends only on $p_i$ and 
$u_i$). Since adjacency in $\Gamma(G)$ factors over Sylow components, 
together these componentwise isomorphisms gives a graph isomorphism from $\Delta(G)$ to $\Delta(H).$ Also, $|\I_G|=|\I_H|,$ hence
$\Gamma(G) \cong \Gamma(H)$.

\medskip
 To prove the forward direction, suppose on the contrary that $S_{p_j}(G)$ is cyclic and $S_{p_j}(H)$ 
is non-cyclic for some prime $p_j.$ Then
$$c_j(G) = p_j^{u_j}; \qquad 
c_j(H) = p_j^{u_j-2}(p_j^2-1) = p_j^{u_j-2}(p_j-1)(p_j+1).$$ Thus $c_j(G)\neq c_j(H).$ Moreover, a mismatch at any single prime forces the following relation:

$$\prod_{i=1}^k c_i(G) \neq \prod_{i=1}^k c_i(H),$$
hence $|\I_G| \neq |\I_H|$. Since the number of isolated vertices is a 
graph invariant, $\Gamma(G) \not\cong \Gamma(H).$

\end{proof}

\section{Spectrum of the graph}\label{sec:6}
In this section we investigate the adjacency and Laplacian spectra of the generating graph of abelian groups.

\subsection{Adjacency Spectrum}
 In this section, we compute the adjacency matrix spectrum of $\Gamma(G)$ for abelian groups $G$. We first determine the spectra for abelian $p$-groups. We denote the adjacency matrix of a graph $\G$ by $A(\G)$. For notational convinence, we take $A$ for the graph $\Gamma(G).$
 \begin{theorem}\label{thm:padjc}
            Let $G=\Z_{p^e}\times  \Z_{p^t},$ where $p$ is a prime. Then the spectrum of $A$ is given by
            $$\mathrm{Spec}(A)=\begin{pmatrix}-\alpha&0&p\alpha\\ p&p^{e+t}-(p+1)&1\end{pmatrix},$$  where $\alpha=\varphi(p^{e+t-1}).$ 
        \end{theorem}
        \begin{proof}
            To determine the nonzero eigenvalues of $A$, it suffices to analyze 
            $\Delta(G)$, since each isolated vertex contributes the eigenvalue $0$. 
    We have 
            \[
            \Delta(G) \cong K_{(p+1)\times \varphi(p^{e+t-1})}.
            \] (see Theorem~\ref{thm:pgraph}).
            The $(p+1)$ parts give an equitable partition of $V(\Delta(G))$, and 
            the corresponding quotient matrix of $A(\Delta(G))$ is of size
            $(p+1)\times(p+1)$, given by 
            $$
            \widetilde{A} = \alpha(J - I), \quad \alpha = \varphi(p^{e+t-1}),
            $$
            where $J$ and $I$ are matrices of size
     $(p+1)\times (p+1)$. The eigenvalues of $\widetilde{A}$ are $\alpha p$ with 
            multiplicity $1$ and $-\alpha$ with multiplicity $p$. By 
            Theorem~\ref{thm:specA}, these are eigenvalues of $A(\Delta(G))$, 
            hence of $A$. The remaining eigenvalues of $A$ are all zero, with 
            multiplicity $|V(\Gamma(G))| - (p+1)$. This completes the proof.
        \end{proof}
            To determine the spectrum of $A$ in the general case, we use the 
            spectral data of $A(\Gamma(S_{p_i}))$ for each prime $p_i$ dividing $|G|$. 
            Let $X = \{1, 2, \ldots, k\}$ index the primes $p_1, p_2, \ldots, p_k$ 
            dividing $|G|$, and denote by $E^{+}$ and $E^{-}$ the subsets of $X$ 
            corresponding to positive and negative eigenvalues of $A(\Gamma(S_{p_i}))'s$, 
            respectively and partitioned the set $X$. Now we will consider the cases when all sylow subgroups are non-cyclic and when at least one is cyclic.

\begin{theorem}\label{thm:Acase}
    Let $G = S_{p_1} \times \cdots \times S_{p_k}$ where each $S_{p_i}$ 
    is a non-cyclic abelian $p_i$-group of order $p_i^{u_i}$. The 
    non-zero adjacency eigenvalues of $\Gamma(G)$ are
    $$\lambda_{E^-} = (-1)^{|E^-|}\varphi\!\left(\frac{|G|}{p_1\cdots 
    p_k}\right)\prod_{j \in E^+} p_j$$
    parametrized over all subsets $E^- \subseteq \{1,\ldots,k\}$ with 
    $E^+ = \{1,\ldots,k\} \setminus E^-$, each with multiplicity 
    $\prod_{i \in E^-} p_i$ or 1 if $E^-$ is empty. The eigenvalue $0$ has multiplicity 
    $|G| - \sigma(p_1 \cdots p_k)$, where $\sigma$ denotes the sum of 
    divisors function.
    \end{theorem}
\begin{proof}
    From Theorem \ref{thm:case1}, we have $$\Ga(G)\cong \otimes_{i=1}^k \Ga(S_{p_i}).$$ 
To determine non-zero eigenvalues of the matrix $A,$ it is suffice to find all the eigenvalues of $A(\Ga(S_{p_i}))$ for each $i.$ From the previous theorem, the spectrum of $\Ga(S_{p_i})$ is given by
\begin{equation*}
    \mathrm{Spec}(A(\Ga(S_{p_i})))=
    \begin{pmatrix}-\varphi(p_i^{u_i-1})&0&p_i\varphi(p_i^{u_i-1})\\ p_i&p_i^{u_i}-(p_i+1)&1\end{pmatrix}.
\end{equation*}
Since corresponding to each $S_{p_i},$ we have three distinct eigenvalues. We denote $\mu_{i,1}<\mu_{i,2}<\mu_{i,3}$ as the eigenvalues of $A(\Ga(S_{p_i}))$ with multiplicity $m_1,m_2$ and $m_3,$ respectively. 
Since $\Gamma(G)$  is the Kronecker product of graphs $\Ga(S_{p_i}),$ the eigenvalues of $\Ga(G)$ are of the form
 $\prod\limits_{i,\,a_j} \mu_{i,\,a_j},$ where $i\in\{1,2,\dots, k\}$ and $ a_j\in \{1,2,3\}$ (using Corollary~\ref{cor:kspec}). For further simplification, let $P=E^-\bigcup E^+$ be a partition of the set $X=\{1,2,\dots,k\},$ with the assumption that one of the parts can be empty. Then any non-zero eigenvalue can be expressed in the following simplified form:
$$\prod\limits_{i,\,a_j} \mu_{i,\,a_j}=\prod_{i\in E^-}\mu_{i,a_1}\prod_{j\in E^+} \mu_{j,a_3}=(-1)^{\left|E^-\right|} \varphi\left(\dfrac{|G|}{p_1p_2\dots p_k}\right)\prod\limits_{j\in E^+}p_j$$
and each with multiplicity  $\prod\limits_{i\in E^-}p_i.$ 
Also note that the number of non-zero eigenvalues of $A$ are exactly equal to $\sigma(n_0),$ which basically depends on what indices are occured in the subset $E^-.$ However, the rest of the eigenvalues are zero. This completes the proof.
\end{proof}

\begin{theorem}\label{thm:Acase2}
    Let $G = S_{p_1}\times\cdots\times S_{p_j}\times\mathbb{Z}_s$, where 
    $S_{p_i}$ is a non-cyclic abelian $p_i$-group of order $p_i^{u_i}$ 
    for $1\leq i\leq j$, and $s = q_1^{t_1}\cdots q_k^{t_k}$ with 
    distinct primes $q_1 < \cdots < q_k$. Set $\bar{s} = \prod_{i=1}^{j}
    p_i^{u_i}$, $\bar{s}_0 = p_1\cdots p_j$, and $s_0 = q_1\cdots q_k$. 
    Let $X = \{1,\ldots,j\}$ and let $(E^+, E^-)$ be a partition of $X$. 
    The non-zero adjacency eigenvalues of $\Gamma(G)$, counted with 
    multiplicity, are of the form
    $$(-1)^{|E^-|}\frac{s}{s_0}\sqrt{\varphi(s_0)}
    \varphi\!\left(\frac{\bar{s}}{\bar{s}_0}\right)\prod_{m\in E^+}
    p_m\prod_{\beta=1}^{k}u_{\beta,\ell_\beta},$$
    parametrized over all partitions $(E^+,E^-)$ of $X$ and all choices 
    $(\ell_1,\ldots,\ell_k)\in\{1,2\}^k$, where
    $$u_{\beta,\ell_\beta} = \frac{\sqrt{\varphi(q_{k+1-\beta})} \pm 
    \sqrt{4+\varphi(q_{k+1-\beta})}}{2}, \qquad \ell_\beta = 1,2.$$
    Each such eigenvalue has multiplicity $\prod_{i\in E^-}p_i$ or 1 if $E^-$ is empty. The 
    eigenvalue $0$ has multiplicity $|G| - 2^k\sigma(\bar{s}_0)$.
    \end{theorem}
\begin{proof}
Recall that $$\Ga(G)=\left(\otimes_{i=1}^{j}\Ga(S_{p_i})\right)\otimes\, \widehat\Ga(\Z_s).$$ 
Let $P= E^- \bigcup E^+$ be a partition the subset $X=\{1,2,\dots, j\}$ which corresponds to negative and positive eigenvalues of $A(\Ga(S_{p_i})),$ where $i\in X.$ From Theorem~\ref{thm:Aloop} and Corollary~\ref{cor:Etensor}, the non-zero eigenvalues of $\widehat A$ of $\widehat\Ga(\Z_s)$ are given by 
$\dfrac{s}{s_0}\sqrt{\varphi(s_0)} \prod\limits_{\beta=j+1}^{k} u_{\beta,\ell_\beta}$ 
parametrized over all choices $(\ell_1,\ldots,\ell_k)\in\{1,2\}^k$, 
where for each $\beta$,
$$u_{\beta,1} = \frac{\sqrt{\varphi(q_{k+1-i})} + \sqrt{4 + \varphi(q_{k+1-\beta})}}{2}, \qquad u_{i,2} = \frac{\sqrt{\varphi
(q_{k+1-\beta})} - \sqrt{4 + \varphi(q_{k+1-\beta})}}{2}$$
are the two eigenvalues of the $i$-th factor matrix $U_i$. Combining Theorem~\ref{thm:Acase} and Corollary~\ref{cor:kspec} we get the required expression. Since there are exacty $2^k\sigma(\bar s_0)$ non-zero eigenvalues of the adjacency matrix, hence the rest are zero. 
\end{proof}
\subsection{Laplacian spectrum}
In this section we determine the spectrum of the Laplacian matrix of abelian groups. We first list the Laplacian eigenvalues of $\Gamma(G),$ of  $p$-groups.

\begin{theorem}
    Let $G$ be a $p$-group of the form $\Z_{p^e}\times \Z_{p^t.}$ Then the Laplacian spectrum is given by
$$\mathrm{Spec}(L)=\begin{pmatrix}0&(p+1)\alpha&p\alpha\\
 1+p^{e+t-2}&p&(p+1)(\alpha-1)\end{pmatrix},$$ 
 where $\alpha=\phi(p^{e+t-1}).$ 
\end{theorem}
\begin{proof}
    To find the non-zero eigenvalues of $L,$ it is sufficient to find the eigenvalues of $L(\Delta(G)).$
    Recall that $$\Delta(G)\cong K_{
    (p+1)\times \varphi(p^{e+t-1})}.$$  Since $$L(\Delta(G))=p\alpha I-A(\Delta(G)).$$ Note that the size of $L(\Delta(G))$ is $(p+1)\alpha\times (p+1)\alpha.$ Clearly, the eigenvalues are $p\alpha-\lambda_i,$ where $\lambda_i$'s are the eigenvalues of $A(\Delta(G))$ (see Theorem~\ref{thm:padjc}). This completes the proof.
\end{proof}
Now let us investigate the general cases.\par
\begin{theorem}
    Let $G=S_{p_1}\times S_{p_2}\times\dots \times S_{p_k},$ where each Sylow $p_i$-subgroup is non-cyclic and $|S_{p_i}|=p_i^{u_i}$.  Then the spectrum of the Laplacian matrix of $\Gamma(G),$ denoted as $\mathrm{Spec}(L),$ is given by
    \begin{equation*}
        \begin{pmatrix}d-\left((-1)^{\left|E^-\right|} \varphi\left(\frac{|G|}{p_1p_2\dots p_k}\right)\prod\limits_{j\in E^+}p_j\right)&d&0\\ \prod\limits_{i\in E^-}p_i&|G|-|\I(G)|-\sigma(p_1p_2\dots p_k)&|\I(G)|+1\end{pmatrix},
    \end{equation*}
    where $d$ is the degree of the graph $\Delta(G).$
    
    \end{theorem}
    \begin{proof}
      By definition $L=D(G)-A.$ We can write it in the block form as follows: 
     $$L=\begin{bmatrix}L(\Delta(G))&\bf 0\\\bf 0&\bf 0\end{bmatrix},$$
     where $0$ is the zero matrix corresponding to the isolated vertices and $\Delta(G)$ is the largest connected component. 
     Since the graph $\Delta(G)$ is $d$-regular, where $d=r_1r_2\cdots r_k,$ $r_i=p_i\alpha_i,$ and $\alpha_i=\phi(p_i^{u_i-1}),$  $L(\Delta(G))=dI-A(\Delta(G)).$ Let $\lambda_i$  be an eigenvalue of $A(\Delta(G)).$ Then, $d-\lambda_i$ are the corresponding eigenvalues of $L(\Delta(G)).$ So, Theorem~\ref{thm:Acase} gives the required expression of all the eigenvalues of $L$. Note that if we consider $E^{-}$ empty, then the corresponding eigenvalue of $L(\Delta(G))$ is zero with multiplicity 1. Thus, we will have exactly $|\I(G)|+1$ zero eigenvalues of $L.$
        \end{proof}

 \begin{rem} 
  Let $G$ has at least one cyclic Sylow $p$-subgroup. 
    We can write $\Ga(G)=H\,\otimes\, \widehat\Ga(\Z_s),$
    where $H=\otimes_{i=1}^{j}\Ga(S_{p_i}).$ Note that $\Delta(H)$ is $d=r_1r_2\cdots r_j$ regular graph. It is clear that, $\widehat \Ga(\Z_s)$ is not a regular graph because $\deg(0)\neq \deg(1).$  We denote $S(L(\widehat\Ga(\Z_s)))=\{u_1,u_2,\dots ,u_s\},$ as the multiset of eigenvalues of $L(\widehat\Ga(\Z_s)).$
    Then we get a subset $\{du_1,du_2,\dots , du_s\}$ of the set of eigenvalues of $L$ (see Theorem~\ref{thm:klap}). 
 \end{rem}

 Let $S(A(\Delta(H)))=\{\lambda_1,\lambda_2,\dots,\lambda_{|V(\Delta(H))|}\}$ denote the multiset of the eigenvalues of the adjacency matrix of $\Delta(H). $  The following result gives the expression for the complete spectrum of $L(\Delta(G))$ in terms of the eigenvalues of the adjacency matrix of $\Delta(G).$ 
    \begin{theorem}\label{thm:lap}
Let $G$ be a group and has at least one cyclic Sylow $p$-subgroup. Then
            \[ S\left(L(\Delta(G))\right)
               = \bigcup_{i=1}^{|V(\Delta(H))|} 
                  S\left(d \widehat{D} - \lambda_i \widehat{A}\right), \]
            where $d=r_1r_2\cdots r_j,$ 
            $\widehat{D}$ is the vertex degree matrix of $\widehat{\Gamma}(\mathbb{Z}_s)$ and  
            $\widehat{A}$ is the adjacency matrix of $\widehat{\Gamma}(\mathbb{Z}_s)$.\index{Spectrum!for $\Z_n\times \Z_m$}
     \end{theorem}
    \begin{proof}
        We consider the connected component $\Delta(G)$ of $\Gamma(G)$. By 
     Remark~\ref{rem:decompose}, $\Delta(G) = \Delta(H)\otimes\widehat{
        \Gamma}(\mathbb{Z}_s)$, and so
        $$L(\Delta(G)) = D(\Delta(H))\otimes\widehat{D} - A(\Delta(H))
        \otimes\widehat{A},$$
        where $\widehat{A}$ and $\widehat{D}$ denote the adjacency and degree 
        matrices of $\widehat{\Gamma}(\mathbb{Z}_s)$ respectively. Since 
        $A(\Delta(H))$ is symmetric, there exists an orthogonal matrix $P$ 
        such that $PA(\Delta(H))P^T = \Lambda_H$, where $\Lambda_H = 
        \mathrm{diag}(\lambda_1,\ldots,\lambda_{|V(\Delta(H))|})$. Setting 
        $M = P\otimes I$, which is orthogonal since $P$ is, we have $M^{-1} 
        = M^T$, so $L(\Delta(G))$ and $M^TL(\Delta(G))M$ are similar and 
        share the same eigenvalues. Computing:
        \begin{align}\label{eq:abbound}
        M^TL(\Delta(G))M = M^T(D(\Delta(H))\otimes\widehat{D})M - 
        M^T(A(\Delta(H))\otimes\widehat{A})M.
        \end{align}
    
        Applying the mixed product property of the Kronecker product, 
        $(X\otimes Y)(U\otimes V) = XU\otimes YV$, to the right hand side 
        of Equation~\eqref{eq:abbound} with $M = P\otimes I$, we get
    \begin{align*}
        M^T(D(\Delta(H))\otimes \widehat D)M&=(P\otimes I)^T(D(\Delta(H))\otimes \widehat D)(P\otimes I)\\
        &=P^TD(\Delta(H))P\otimes \widehat D\\
        &=D(\Delta(H))\otimes \widehat D=d I\otimes \widehat D,
    \end{align*}
    \begin{align*}
      \text{and}\quad  M^T(A(\Delta(H))\otimes \widehat A)M&=(P\otimes I)^T(A(\Delta(H))\otimes \widehat A)(P\otimes I)\\
        &=P^TA(\Delta(H))P\otimes \widehat A=\Lambda_H\otimes \widehat A\\
        &=\begin{bmatrix}
            \la_1\widehat A&\bf 0&\dots &\bf 0\\
            \bf 0&\la_2\widehat A&\dots &\bf 0\\
            \vdots&\vdots&\ddots&\vdots\\
        \bf 0&\bf 0&\dots&\la_{|V(\Delta(H))|}\widehat A
        \end{bmatrix}.
    \end{align*}
    On combining the expressions, we get
    $$M^TL(\Delta(G)) M=\begin{bmatrix}
        d\widehat D-\la_1\widehat A&\bf 0&\dots &\bf 0\\
    \bf 0&d\widehat D-\la_2\widehat A&\dots &\bf 0\\
    \vdots&\vdots&\ddots&\vdots\\
    \bf 0&\bf 0&\dots&d\widehat D-\la_{|V(\Delta(H))|}\widehat A \end{bmatrix}.$$
    Thus $S(L(\Delta(G)))=S(M^TL(\Delta(G))M)=\bigcup\limits_{i=1}^{|V(\Delta(H))|}S(d \widehat D-\la_i\widehat A).$
    \end{proof}

\subsection*{Availability of data and materials} Data sharing does not apply to this article as no data sets
were generated or analysed during the current study.
\subsection* {Competing interests} The author declare that they have no competing interests.

\addcontentsline{toc}{section}{Bibliography}
\bibliographystyle{plain}
\bibliography{Abelian}

 \end{document}